\documentclass[11pt]{article}
\usepackage{amsmath,amssymb,amstext,dsfont,fancyvrb,float,fontenc,graphicx,subfigure, theorem,pseudocode,xcolor}

\title{\Large\bf Two notes on generalized Darboux properties and related features of additive functions}
\author{\sc Gabriel Istrate}
\date{}

\cleardoublepage 
\usepackage[strict]{changepage}
\def\abstractname{Abstract -}   
\def\abstract{\begin{adjustwidth}{1cm}{1cm} \par    \footnotesize \noindent {\bf \abstractname} 
\def\endabstract{ \end{adjustwidth} \smallskip }}

{\theorembodyfont{\itshape}\newtheorem{theorem}{Theorem}[section]}
{\theorembodyfont{\itshape}\newtheorem{proposition}{Proposition}[section]}
{\theorembodyfont{\itshape}\newtheorem{definition}{Definition}[section]}
{\theorembodyfont{\itshape}\newtheorem{lemma}{Lemma}[section]}
{\theorembodyfont{\itshape}\newtheorem{corollary}{Corollary}[section]}
{\theorembodyfont{\rm}}
{\theorembodyfont{\rm}}
{\theorembodyfont{\rm}}
{\theorembodyfont{\rm }}

\def\goesto{\rightarrow}
 
\begin{document}
\maketitle
\vskip 1.5em
 \begin{abstract}
 We present two results on generalized Darboux properties of additive  real functions. 
 
The first results deals with a weak continuity property, called {\bf Q}-continuity, shared by {\em all} additive functions. We show that every {\bf Q}-continuous function is the uniform limit of a sequence of Darboux functions. The class of {\bf Q}-continuous functions includes the class of Jensen convex functions. We discuss further connections with related concepts, such as {\bf Q}-differentiability. 

Next, given a ${\bf Q}$-vector space $A\subseteq {\bf R}$ of cardinality ${\bf c}$ we consider the class ${\cal DH}^{*}(A)$ of additive functions such that for every interval $I\subseteq {\bf R}$, $f(I)=A$. We show that every function in class ${\cal DH}^{*}(A)$ can be written as the sum of a linear (additive continuous) function and an additive function with the Darboux property if and only if $A={\bf R}$. We apply this result to obtain a relativization of a certain hierarchy of real functions to the class of additive functions. 
 \end{abstract}
 

\begin{MSC}
Primary: 26A15. 
\end{MSC}

\section{Introduction}

Structural properties (such as weak notions of continuity, Darboux property or Jensen convexity) of real functions, especially properties common to large classes of functions, 
 have formed a significant part of the interests of Professor Solomon Marcus in Real Analysis \cite{marcus-hamel,marcus-jensen,marcus-darboux-connected,marcus-darboux-tams, marcus-quasi}. This paper presents two contributions in this area, motivated by the Darboux-like properties of additive functions. These functions are the solutions of the Cauchy functional equation
\[
f(x+y)=f(x)+f(y)
\]
(see \cite{kuczma} for a modern introduction). Together with  researchers such as A.M. Bruckner, J.~Ceder, F.B. Jones,  W.~Kulpa, F.~Obreanu, J. Sm\'{i}tal, M.~Weiss, S.~Marcus \cite{marcus-hamel} investigated structural properties of additive functions, such as the connectedness of the graph and the Darboux property. 

Besides the linear functions $f(x)=r\cdot x$ for some $r\in {\bf R}$, the class of additive functions contains a host of functions lacking most ``regularity" properties of linar functions: these additive functions are nowhere continuous, nowhere monotone and have a graph that is dense in ${\bf R}\times {\bf R}$. As noted in \cite{kuczma} (pp. 322) ``discontinuous additive functions have many pathological properties. Therefore it is often believed that such functions cannot have any nice property." But this is not the case: since the writing of  \cite{Ist3}, a Ph.D. thesis \cite{ban-thesis} and several papers (among them \cite{BAN95,BAN96,cies,cies2,cies3,natkaniec,natkaniec-rosen}) have investigated Darboux-like properties and relaxations of continuity for additive functions. 

The first contribution of this paper is the introduction of a weak continuity property, we call {\bf Q}-continuity (here {\bf Q} refers to the set of rationals) shared by {\em all} additive functions. We show that every {\bf Q}-continuous function is the uniform limit of a sequence of Darboux functions. The class of {\bf Q}-continuous functions includes the class of Jensen convex functions. We further discuss the related notion of {\bf Q}-differentiability of a real function.  

In the second note we consider the class ${\cal DH}^{*}(A)$ (where $A$ is a vector subspace of {\bf R} of cardinality $c$.) of additive functions such that for every interval $I\subseteq {\bf R}$, $f(I)=A$. We show that every function in class 
${\cal DH}^{*}(A)$ can be written as the sum of a linear (additive continuous) function and an additive function with the Darboux property only if $A={\bf R}$. The application of this result to the relativization of a function hierarchy to the class of additive functions is presented. 

Results in this paper originate in my graduation thesis, written in 1994 under the supervision of Professor Marcus \cite{Ist3}. Rather than just publishing a set of old results, I have attempted, however to reconsider them from the vantage point of a more senior researcher (though one no longer actively working in Real Analysis): I have substantially modified my original definitions from \cite{Ist3}, reproved some results and added new ones, and situated them in the context of more recent literature in Real Analysis. It is my hope (particularly with respect to the results in Section~\ref{s2}) that the notions presented here can stimulate further research.  

The paper concludes with a recollection on my collaboration with professor Marcus and its students that led to these results. 

\section{{\bf Q}-continuous functions and the inclusion ${\cal H}\subset {\cal U}$.}
\label{s2}
Let ${\cal H}$ be the class of additive functions, 
and let ${\cal U}$ be the class of functions that are the uniform limit of a sequence 
of Darboux functions. Bruckner, Ceder and Weiss \cite{BCW} have observed that the inclusion 
${\cal H}\subset {\cal U}$ is true, and follows from a result of Beckenbach and Bing 
\cite{Bebi} 
(every Jensen convex function is in ${\cal U}$). That is, the following 
hierarchy holds: 
\begin{equation}\label{hier}
{\cal H}\subset {\cal J}\subset {\cal U},
\end{equation} 

where ${\cal J}$ is the set of Jensen continuous function. On the other hand linear functions 
are, evidently continuous, and continuous functions have the Darboux property. We wondered whether 
these results ``relativize" (under suitable weak notions of continuity and Darboux property) to all 
additive functions.  That is, we wondered whether there exists a weak notion of continuity
such that
\begin{itemize}
\item[(i). ] every additive function is ``weakly continuous". 
\item[(ii). ] every ``weakly continuous"  function belongs to the class ${\cal U}$.
\end{itemize}

\noindent i.e. the inclusion ${\cal H}\subset {\cal U}$ is the 
consequence of some weak analogue of Darboux' Theorem ? 

This question motivated work (presented below) in Section 5.4 our bachelor thesis \cite{Ist3}. Since then, the related question ``are additive functions continuous in some weak sense" has been independently raised (according to \cite{kostyrko}) by T.~Sal\'at.

In the sequel we present a weak continuity notion with properties (i). and (ii), somewhat modifying the solution from 
\cite{Ist3}:  

\begin{definition}

Let $x\in {\bf R}$. Function $f:{\bf R}\rightarrow {\bf R}$ is {\bf Q}-{\em continuous at $x$ from the left} if there exists $\epsilon > 0$ such that for every $y,z$ with $x-\epsilon<y< z\leq x\in {\bf R}$, 
there exists a continuous function $\overline{f}_{y,z}:[y,z]\rightarrow {\bf R}$ such that 
\begin{equation}
\overline{f}_{y,z}=f\mbox{ on the set }\{\alpha y+ (1-\alpha)z: \alpha \in {\bf Q}\cap [0,1]\}
\label{qc-cond}
\end{equation}
{\bf Q}-{\em continuity at $x$ from the right} is defined analogously. 
{\bf Q}-{\em continuity at $x$} is defined simply as {\bf Q}-continuity at $x$ both from the left and from the right.
\label{qc}
\end{definition}

We stress that in the above definition ${\bf Q}$ refers to the class of 
rational numbers. In particular our definition is to be distinguished 
from the (shorthand for the) class of quasicontinuous  functions \cite{quasicontinuous}, often abbreviated in the literature as q-continuous. 

Definition~\ref{qc} should be compared to Banaszewski's $\mathcal{E}$-continuity \cite{BAN95}, and to the similar notion of path Darboux property \cite{path-darboux}. 
The former definition requires continuity via  a system of paths, a feature shared (in a very restrictive form) by all 
{\bf Q}-continuous functions. However, our definition is substantially more ``rigid": we need, in fact, an uncountable system of paths. 

The following result provides a substantial number of examples of {\bf Q}-continuous functions: 

\begin{theorem} The following are true: 
\begin{itemize} 
\item[(a). ]If $f,g:{\bf R}\rightarrow {\bf R}$ are functions that are (left, right, bilaterally) {\bf Q}-continuous at point $x_{0}$ then for all $\alpha \in {\bf R}$, 
$f+g,\alpha\cdot f, f\cdot g$ are (left, right, bilaterally) {\bf Q}-continuous at $x_{0}$. 
\item[(b). ]All additive functions $H$ are {\bf Q}-continuous everywhere. 
\end{itemize} 
\label{first}
\end{theorem} 
\begin{proof} 

Immediate. For $f+g$ we take $\epsilon_{f+g}=\min\{\epsilon_{f},\epsilon_{g}\}$ for the parameter in equation~(\ref{qc-cond}) 
and $\overline{f+g}_{y,z}=\overline{f}_{y,z}+\overline{g}_{y,z}.$ For $f\cdot g$, given $\epsilon$ suitable for both $f$ and $g$, for any $x-\epsilon\leq y<z\leq x$, by additivity the component functions $f,g$ have continuous extensions $f_{y,z},g_{y,z}$ with the properties in equation~(\ref{qc-cond}) on $[y,z]$. It is easy to see that $f_{y,z}\cdot g_{y,z}$ is such an extension for $h$.   

The second part is trivial, since  for $\alpha\in {\bf Q}\cap [0,1]$, $H(\alpha\cdot y+(1-\alpha)z)=\alpha\cdot H(y)+(1-\alpha)H(z)$.
\end{proof} 

\begin{lemma} 
If $f:{\bf R}\rightarrow {\bf R}$ is (lower/upper/bilaterally) {\bf Q}-continuous at $x_{0}$ and 
$f(x)=0$ locally (below/above/around) $x_{0}$, then $f(x_{0})=0$. 
\label{second}
\end{lemma} 
\begin{proof}
Clearly $f_{y,z}=0$ for close enough points $y,z$. Just choose a pair such that $x_{0}\in \{\alpha y+ (1-\alpha)z: \alpha \in {\bf Q}\cap [0,1]\}$. 
\end{proof} 

It will be also useful to introduce a weaker version of the class ${\cal U}$, 
denoted ${\cal U}_{0}$ (see \cite{El,Ma,BCW}) that admits a ``local'' characterization. 
\begin{definition}
\begin{enumerate}
\item ${\cal U}_{0}$ is the class of functions 
$f:I \longrightarrow {\bf R}$ such that for every subinterval $J \subset I$, 
the set f(J) is dense in \[[inf_{x \in J}(f(x)),\\
 sup_{x \in J}f(x)].\]
\item Define 
\[
C_{0}^{+}(f,x)=\{ y \in \overline{\protect {\bf R}}|\forall M \in V(x_0),
\forall z>y, card(f^{-1}(M) \cap (y,z])\neq \emptyset\}.
\]

Similarly (but working with intervals upper bounded by $y$) one defines the limit set $C_{0}^{-}(f,x)$. 
\item $f:I \longrightarrow {\bf R}$ is called {\em locally--${\cal
U}_{0}$ at point $x$} if the limit sets ${C_0}^{-}(f,x)$, $
{C_0}^{+}(f,x)$ are intervals.
\end{enumerate}
\end{definition}

\begin{definition} 
Let $I\subset {\bf R}$ be an interval.  A function $f:I\rightarrow {\bf
R}$ is {\em Jensen convex} if, for all $x,y\in I$,
\[
f(\frac{x+y}{2})\leq \frac{f(x)+f(y)}{2}.
\]
The class of Jensen convex real functions will be denoted by $\mathcal{J}$. Clearly $\mathcal{U}\subset \mathcal{J}$. 
\end{definition} 

The above-mentioned ``local'' characterization of class
${\cal U}_{0}$ is

\begin{proposition}\cite{BCW}
$f$ belongs to the class ${\cal U}_{0}$ if and only if it is
locally--${\cal U}_{0}$ at every point $x$.
\end{proposition}

With this definition we have the following result, that answers our question: 

\begin{theorem} The following results hold: 
\begin{enumerate}
\item Every Jensen convex function is {\bf Q}-continuous. 
\item If $f$ is {\bf Q}-continuous at $x$, then it is locally-${\cal U}_{0}$ 
at $x$. 
\item Every {\bf Q}-continuous function belongs to the class ${\cal U}$. 
\end{enumerate}
\end{theorem}
In other words, we refine hierarchy~\ref{hier} to: 
\begin{equation} 
{\cal H}\subset {\cal J}\subset {\bf QC}\subset {\cal U},
\end{equation} 

\begin{proof}
\begin{enumerate}
\item A result from \cite{Bebi} (see also comments in the proof of Theorem~4.5 from \cite{BCW}) shows that condition~(\ref{qc-cond}) is true, whenever $f$ is a Jensen convex function, for {\em every} $y<z$. 

\item We will show that $C^{-}(f,x)$ is an interval (a similar result will hold for $C^{+}(f,x)$). Indeed, let $a<b\in 
C^{-}(f,x)$. Let $\xi \in (a,b)$ and consider two sequences $x_{n},y_{n}$, $n\geq 1$, converging to $x$ from below, such that 
\[
\lim_{n\rightarrow \infty} f(x_{n})=a,
\]
\[
\lim_{n\rightarrow \infty} f(y_{n})=b.
\]
Without loss of generality we may assume that $x_{1}<y_{1}<x_{2}<y_{2}<\ldots < x_{n}<y_{n}<\ldots < x$. Since $f$ is {\bf Q}-continuous at $x$, for large enough $n$ there exists continuous function $g_{n}:\stackrel{def}{=} g_{x_{n},y_{n}}:[x_{n},y_{n}]\rightarrow {\bf R}$ such that $f=g_{n}|_{\{\alpha x_{n}+ (1-\alpha)y_{n}: \alpha \in {\bf Q}\cap [0,1]\}}$. 
Since $g_{n}$ is continuous, for large enough $n$ there exists $z_{n}\in [x_{n},y_{n}]\cap \{\alpha x_{n}+ (1-\alpha)y_{n}: \alpha \in {\bf Q}\cap [0,1]\}$ such that 
\[
|g_{n}(z_{n})-\xi |\leq \frac{1}{n}
\]
Since $f(z_{n})=g_{n}(z_{n})\rightarrow \xi$ (as $n\rightarrow \infty$), we infer that $C^{-}(f,x)$ is an interval. 
\item 
Let us now consider an open interval $N$ so that $f^{-1}(N)\neq
\emptyset \mbox{ and } x_{0}\in f^{-1}(N)$. Let $y_{0}=f(x_{0})\mbox{
and }\epsilon >0 \mbox{ so that }(y_{0}-\epsilon, y_{0}+\epsilon
)\subset N$. For every $z\in {\bf R}$ there exists $\alpha \in {\bf
Q}^{*}$ such that $|f(x_{0}+\alpha z)-\xi |<\epsilon $. It follows
that $x_{0}+\alpha z\in f^{-1}(N)$, therefore $f^{-1}(N)$ is {\bf
c}-dense in itself. By applying the characterization of class $\mathcal{U}$ (Theorem 3.2 in \cite{BCW}) we infer that $f\in {\cal
U}$.
\end{enumerate}
\end{proof}

{\bf Q}-continuous functions have many other properties that are 
reminiscent of continuous functions. We give next an example of such property:  

\begin{definition}
Function $f:{\bf R} \goesto {\bf R}$ is {\bf Q}-{\em differentiable at
$x_{0}$} if there exists $\lambda \in {\bf R}$ such that function
\[
g_{x_{0}}(x)= \left \{\begin{array}{ll}
                \frac{f(x)-f(x_{0})}{x-x_{0}}
                &\mbox{, for }x\neq x_{0} \\
             \lambda 
             &\mbox{, for }x=x_{0}, 
             \end{array}
      \right.
\] 

is {\bf Q}-continuous at $x_{0}$. The value $\lambda$ is called {\em
the {\bf Q}-derivative of $f$ at $x_{0}$} (we will write $\lambda =
f^{\prime}_{{\bf Q}}(x_{0})$)
\label{q-deriv}
\end{definition} 

\begin{theorem} 
If $f$ is {\bf Q}-derivable at $x_{0}$ then $f$ is {\bf Q}-continuous at $x_{0}$. 
\end{theorem} 

\begin{proof}
We write the equation above as 
\[
f(x)=f(x_{0})+(x-x_{0})\cdot g_{x_{0}}(x)
\]
and then use the fact that constants are {\bf Q}-continuous, $x-x_{0}$ is {\bf Q}-continuous and Lemma~\ref{first}. 
\end{proof} 

\begin{theorem} 
The value $f^{\prime}_{{\bf Q}}(x_{0})$ of a function that is {\bf Q}-differentiable at $x_{0}$ is well defined (that is, there exists at most one completion value $\lambda$). 
\end{theorem} 
\begin{proof} 
Suppose there exist two completions $\lambda,\mu$.  By Lemma~\ref{first} 
the difference of the two functions $g_{x_{0}}$ is a {\bf Q}-continuous that is locally zero around $x_{0}$. By Lemma~\ref{second} it must be that $\lambda=\mu$. 
\end{proof} 

Our extension of notions of continuity/differentiability is not as strong as one might believe: {\bf no} additive function is  {\bf Q}-differentiable, other than the linear functions. 
\begin{theorem} 
Consider an additive function $f$ that is {\bf Q}-differentiable at $x_{0}=0$ and let $r=f^{\prime}_{{\bf Q}}(x_{0}).$ Then $f(x)=rx$ on an open neighborhood of zero. 
\end{theorem} 
\begin{proof} 
Consider $x$ sufficiently close to zero. Then function $g_{0}$ is {\bf Q}-continuous at zero and constant on the set 
$\{\alpha y: \alpha \in {\bf Q}\}$. By the definition of {\bf Q}-continuity, with  $y=x,z=0$ we infer that $g_{0}(x)=r$, i.e. $f(x)=rx$ around zero. 
\end{proof}

Kostyrko \cite{kostyrko} has employed the terminology ``type A property" to refer to continuity properties valid in the class of additive functions for linear functions only, and ``type B" for properties that are not of type A. Extending this language one could say that {\bf Q}-differentiability of functions is a ``locally type A property", while {\bf Q}-continuity is of type B. 

One way to turn {\bf Q}-differentiability into a type B property is to relax 
Definition~\ref{q-deriv}: rather than using as benchmarks for differentiability the linear functions, {\it we will use instead the set of all additive functions:}

\begin{definition} Given real function $f$ and additive function $H$ and $x_{0}\in {\bf R}$, we say that function $f$ is {\bf Q}-{\em differentiable at
$x_{0}$ with respect to $H$} if there exists $\lambda \in {\bf R}$ and a function $F_{x_{0}}$ with $F_{x_{0}}(x_{0})=\lambda$, {\bf Q}-continuous at $x_{0}$
such that function
\[
f(x)=f(x_{0})+[H(x)-H(x_{0})]\cdot F_{x_{0}}(x). 
\] 

The value $\lambda$ is called {\em the {\bf Q}-derivative of $f$ at $x_{0}$ with respect to $H$} (we will write $\lambda =
f^{\prime}_{{\bf Q},H}(x_{0})$)
\label{q-deriv2}
\end{definition} 

This change enables the beginnings of an (intriguing) differential calculus with respect to additive functions:  

\begin{theorem} 
If functions $f,g$ are {\bf Q}-differentiable at $x_{0}$ with respect to additive function $H$ and $\alpha\in {\bf R}$, then $f+g, \alpha\cdot f$ and 
$f\cdot g$ are differentiable at $x_{0}$ with respect to $H$ and 
\begin{align}
(f+g)^{\prime}_{{\bf Q},H}(x_{0})=f^{\prime}_{{\bf Q},H}(x_{0})+g^{\prime}_{{\bf Q},H}(x_{0})\\
(\alpha f)^{\prime}_{{\bf Q},H}(x_{0})=\alpha \cdot f^{\prime}_{{\bf Q},H}(x_{0})\\
(fg)^{\prime}_{{\bf Q},H}(x_{0})=f(x_{0})g^{\prime}_{{\bf Q},H}(x_{0})+
g(x_{0})f^{\prime}_{{\bf Q},H}(x_{0})\label{prod}
\end{align} 

\end{theorem}
\begin{proof} 
If 
\[
f(x)=f(x_{0})+[H(x)-H(x_{0})]\cdot F_{x_{0}}(x). 
\] 
and 
\[
g(x)=g(x_{0})+[H(x)-H(x_{0})]\cdot G_{x_{0}}(x). 
\] 
(where $F_{x_{0}},G_{x_{0}}$ are {\bf Q}-continuous at $x_{0}$, then 
\[
(f+g)(x)=(f+g)(x_{0})+[H(x)-H(x_{0})]\cdot (F_{x_{0}}+G_{x_{0}})(x) 
\] 
and 
\[
(\alpha\cdot f)(x)=\alpha \cdot f(x_{0})+[H(x)-H(x_{0})]\cdot (\alpha \cdot F_{x_{0}})(x). 
\] 
Finally
\begin{align}
(fg)(x) - f(x_{0})g(x_{0})  & =[H(x)-H(x_{0})]\cdot [f(x_{0})G_{x_{0}}(x) + f(x_{0})G_{x_{0}}(x) \nonumber\\
&  +[H(x)-H(x_{0})]\cdot F_{x_{0}}(x)\cdot G_{x_{0}}(x)]
\end{align} 
so let
\[
L(x)= f(x_{0})G_{x_{0}}(x) + f(x_{0})G_{x_{0}}(x) +[H(x)-H(x_{0})]\cdot F_{x_{0}}(x)\cdot G_{x_{0}}(x)
\] 
We have 
\[
(fg)(x) - f(x_{0})g(x_{0}) =[H(x)-H(x_{0})]\cdot L(x)
\]
and 
\[
L(x_{0})=f(x_{0})g^{\prime}_{{\bf Q},H}(x_{0})+
g(x_{0})f^{\prime}_{{\bf Q},H}(x_{0})
\]
therefore relation~(\ref{prod}) follows.
\end{proof}  

Even definition~(\ref{q-deriv2}) does not enlarge too much, though, the class of $H$-differentiable functions: 
\begin{theorem} Given two additive functions the following are equivalent: 
\begin{enumerate} 
\item $H_{1}$ is {\bf Q}-differentiable w.r.t. $H_{2}$. 
\item There exists a constant $r$ such that $H_{1}=rH_{2}$. 
\end{enumerate}
\label{diff}
\end{theorem} 
\begin{proof} 
The reverse implication is easy. So let's deal with the direct one: 
suppose $H_{1}$ is differentiable w.r.t. $H_{2}$ at $x_{0}=0$, and define 
$r=(H_{1})^{\prime}_{{\bf Q},H_{2}}(0)$. We will show that $H_{1}=r\cdot H_{2}$. 
From differentiability we infer
\[
H_{1}(x)=H_{2}(x)\cdot g_{0}(x)
\]
where function $g_{0}$ is {\bf Q}-continuous at $x_{0}=0$. Consider,  $x\in R$ (assume w.l.o.g. $x<0$) such that $H_{2}(x)\neq 0$, and let $\alpha\in {\bf Q}$ such that relation~(\ref{qc-cond}) holds for function $g_{0}$ and parameters $y=\alpha\cdot x$, $z=0$.

Since $H_{1}(\beta\cdot x)= \beta\cdot H_{1}(x)$ and similarly for $H_{2}$, we infer that for $\beta\in [0,\alpha]\cap {\bf Q}$ and continuous extension $\overline{g}_{[0,y]}$ of $g_{0}$
\begin{equation} 
\overline{g}_{[0,y]}(\beta\cdot x)=\frac{H_{1}(x)}{H_{2}(x)}
\label{foo} 
\end{equation} 
Taking the limit $\beta\rightarrow 0, \beta\in {\bf Q}$ in equation~(\ref{foo}) above we infer 
\[
\frac{H_{1}(x)}{H_{2}(x)}=(H_{1})^{\prime}_{{\bf Q},H_{2}}(0)=r.
\]

\end{proof} 

\section{Universally bad additive functions and the additive analogue of a hierarchy}

In \cite{Ist2} we have studied the class ${\cal C}+{\cal D}$, of
functions that are the sum of a continuous and a Darboux function,
with the main purpose of understanding its relationship with the class
${\cal U}$, of functions that are the uniform limit of a sequence of
Darboux functions. It is known that ${\cal C}+{\cal D}\subset {\cal
U}$, and the inclusion is strict. It is also known \cite{Ist1} that
${\cal U}$ is closed under {\em quasi-uniform}
(Arzel\'{a}-Gagaeff-Alexandrov) convergence, defined as follows: 

\begin{definition} 
Let $(X,\rho )$ and $(Y, \sigma )$ be two metric spaces, and let 
$f, f_{n}:X\rightarrow Y$ be functions. Sequence $(f_{n})_{n\geq 1}$ is said to {\em quasi-uniformly converge} to $f$ iff
\begin{itemize}
\item [(i)] $f_{n}$ converges pointwise to  $f$.
\item [(ii)] For every $\epsilon >0$ there exists a sequence (possibily finite)
 of indices $n_{1}<n_{2}<\ldots <n_{p}<\ldots $ and a corresponding sequences of open sets
$G_{1}<G_{2}<\ldots <G_{p}<
\ldots $ s.t. $X=\bigcup _{i} G_{i}$ and for all $i\geq 1$,
$$ x\in G_{i}\Rightarrow \sigma (f(x), f_{n_{i}}(x))<\epsilon .$$ We will employ the notation
$f_{n} \stackrel{AGA}{\rightarrow } f$.
\end{itemize}
\end{definition}

Hence ${\cal U}$ it is also the
closure of ${\cal D}$ under this type of convergence. If we denote by
$U\cdot {\cal A},QU\cdot {\cal A}$ the closure of a class ${\cal A}$
under uniform (quasi-uniform) convergence, the above mentioned result
reads
\begin{equation}
\label{hierarchy}
{\cal C}+{\cal D}\subset \mathcal{U}=U\cdot {\cal D}=QU\cdot {\cal D}.
\end{equation}

An interesting variation on these classes considers an additional
restriction on the functions involved: being {\em additive}. We will
denote by ${\cal H}$ the class of additive functions  and, for a class of functions ${\cal A}$, by ${\cal
AH}$ the class ${\cal A}\cap {\cal H}$.

In the sequel we study the analogue of hierarchy~(\ref{hierarchy}) when
the extra constraint of additivity is imposed. The resulting hierarchy 
does {\em not} simply mirror~(\ref{hierarchy}), if only for the
following reason: as we have seen in the previous section, ${\cal H}\subset {\cal U}$, so in fact
${\cal UH}={\cal H}$. On the other hand this class is easily seen {\em
not} to be equal to $U\cdot {\cal DH}$ (since uniform convergence of
additive functions is trivial). Finally, the comparison with $QU\cdot
{\cal DH}$ is nontrivial.

\begin{definition}
For $A \subset {\bf R}$ denote by ${\cal D}^{*}(A)$ the class of
functions $f:{\bf R}\longrightarrow {\bf R}$ such that for every
interval $I\subset {\bf R}$, $f(I)=A$.
\end{definition}

For instance ${\cal DH}={\cal D}^{*}({\bf R})$. Indeed, the graph of every additive function is dense in ${\bf R}\times {\bf R}$. If it has the Darboux property then it must take every possible value in every interval. 

We provide below examples of ``universally bad" functions in the class of additive functions, in a manner similar in spirit to results in \cite{BAN96},\cite{Ist2}, \cite{natkaniec-kircheim}: 

\begin{theorem} 
If $A\subsetneq {\bf R}$ is a vector space with $|A|={\bf c}$ then
$\emptyset \neq {\cal DH}^{*}(A)\not \subset {\cal CH}+{\cal DH}.$ 
\label{two}
\end{theorem}
\begin{proof} 
Let $\{r_{\alpha}\}_{\alpha <{\bf c}}$  be an enumeration
of ${\bf R}^{*}$, where $\alpha $ ranges over all countable ordinals, and $g_{\alpha}(x)=r_{\alpha}\cdot x \in {\cal CH}$ be an enumeration of nontrivial linear functions. Finally, let $\{U_{i}\}_{i<\omega }$ be a basis for the usual topology on {\bf R}.

Let
$H=\{h_{\alpha}\}_ {\alpha <{\bf c}}$ be a Hamel base, dense in {\bf
R}. The existence of such a base follows easily: starting from a Hamel basis $H_{1}$, if $H_{1}$ is not already dense then we can modify its terms by subtracting rational numbers so that at least one term falls into each $U_{i}$.

Let
$B=\{x_{\alpha}\}_{\alpha <{ \bf c}}$ be a basis for A (it is clear
that $|B|={\bf c}$). 

Define inductively sequences $y_{\alpha}
\in
{\bf R}, z_{\alpha}\in B, t_{\alpha, i}\in H\cap U_{i}, 
\alpha <{\bf c}, i<\omega $ so that 
\begin{equation}
t_{\alpha,i}=t_{\gamma,j} \Rightarrow (\alpha,i)=(\gamma,j)
\label{f}
\end{equation}
\begin{equation}
y_{\alpha}\not \in <x_{\beta}-g_{\alpha}(t_{\beta ,i}), z_{\beta}-g_{\alpha}
(h_{\beta})>_{\beta, i}
\label{s}
\end{equation}

The construction is presented in the transfinite algorithm below. 

\begin{pseudocode}[shadowbox]{Construction of }{y_{\alpha}, z_{\alpha},t_{\alpha,i}}
\mbox{Assume we have defined }t_{\beta , i}, z_{\beta}, y_{\beta}, 
\forall \beta <\alpha , \forall i<\omega.\\
\\
\mbox{\bf \underline{The construction of $y_{\alpha}$:}}\\
\\
\mbox{Since the {\bf Q}-vector space }
<x_{\beta}-g_{\alpha}(t_{\beta ,i}), z_{\beta}-g_{\alpha}(h_{\beta})>_
{\beta<\alpha, i<\omega}\\
\mbox{ is countable, we can chose  }y_{\alpha}\mbox{ outside this space.} \\
\\
\mbox{\underline{\bf The construction of $t_{\alpha ,j}, (j<\omega)$:}}\\
\\
\mbox{ Suppose we have defined }t_{\alpha ,k}, \forall k<j. \\
\\
\mbox{ The set }  A_{j}=\{t\in H\cap U_{j}|x_{\alpha}-g_{\alpha}(t)\in 
<x_{\beta}-g_{\alpha}(t_{\beta ,k}), \\ z_{\beta}-g_{\alpha}(h_{\beta})
, y_{\beta}, z_{\beta},\mbox{\newline} y_{\alpha}, x_{\alpha}-g_{\alpha}(t_{
\alpha ,i})>_{\beta <\alpha, i<j, k<\omega}\}
\\
\mbox{ is at most countable }
\mbox{(since the {\bf Q}-vector space in its definition }\\
\mbox{ is at most countable as well), therefore we can choose }\\ 
t_{\alpha ,j}\in
(H\cap U_{j})\backslash A_{j}.\\
\\

\mbox{\bf \underline{The construction of $z_{\alpha}$:}}\\
\\
\mbox{ The set } 
B_{\alpha}:=\{z\in B |z-g_{\alpha}(h_{\beta})\in 
<x_{\eta}-g_{\alpha}(t_
{\eta ,k}), \\ z_{\beta}-g_{\alpha}(h_{\beta}), y_{\eta}, 
z_{\beta}>_{k<\omega ,\eta \leq \alpha ,\beta <\alpha }\}
\mbox{ is at most countable,}\\ 
\mbox{ therefore (since B has cardinal {\bf c})} 
\mbox{ there exists }z_{\alpha}\in B\backslash B_{\alpha}.\\

\end{pseudocode}

(\ref{f}) results from our choice of $t_{\alpha ,i}$.

To prove (\ref{s}), assume $y_{\alpha}$ is a finite linear combination of
elements from the vector space on the right side of 
(\ref{s}).

\begin{itemize} 
\item[(i). ] if all elements on the right-hand side of (\ref{s}) have rank less that $\alpha$
we obtain a contradiction with our choice of  $y_{\alpha}$.

\item[(ii). ] otherwise,  let $\gamma >\alpha$ be the maximal rank in the linear combination. Considering 
the maximal rank term (one  of $z_{\gamma}-
g_{\alpha}(h_{\gamma}), x_{\gamma}-g_{\alpha}(t_{\gamma ,i}), \mbox{ for some }i<\omega $) in the linear decomposition defining $y_{\alpha}$. 

If it is the first term, then by turning the inequality around, we can obtain $z_{\gamma}-g_{\alpha}(h_{\gamma})$ as a finite linear combination with rational coefficients of $y_{\alpha}$, $x_{\gamma}-g_{\alpha}(t_{\gamma ,i})$ and lower order terms on the right-hand side of~(\ref{s}). But this contradicts the way we defined $z_{\gamma}$. 

Suppose term $z_{\gamma}-g_{\alpha}(h_{\gamma})$ does not appear in the finite linear combination defining $y_{\alpha}$ and, instead, the maximum rank term is $x_{\gamma}-g_{\alpha}(t_{\gamma ,i})$ (for some $i$). Turning the inequality around we write $x_{\gamma}-g_{\alpha}(t_{\gamma ,i})$ as a finite linear combination with rational coefficients of $y_{\alpha}$ and lower order terms on the right-hand side of~(\ref{s}). But this contradicts the way we defined $t_{\gamma,i}$. 
 
\end{itemize} 

Define $h:H\rightarrow {\bf R}$ by
\begin{equation}
h(x)=\left \{\begin{array}{ll}
             x_{\beta},  &\mbox{ if }x=t_{\beta ,i}\mbox{ with }\beta <
             {\bf c}, i<\omega \\
             z_{\beta},  &\mbox{ if }x=h_{\beta} \not \in \{t_{\gamma ,
             i}|\gamma <{\bf c}, i<\omega \}
             \end{array}
      \right.
\end{equation}
It is easily seen that $h(H)=B$ and $h$ is not injective. Extending by linearity $h$ to 
all of {\bf R},  we obtain 
a  function $h\in {\cal D}^
{*}{\cal H}(A)$. 

Supppose that function $h$ can be written as a sum $h=f_{1}+f_{2}\mbox{ with }f_{1}\in
{\cal CH}, f_{2}\in {\cal DH}$.

{\bf Case 1:}

If $f_{1}$ were the zero function then it would follow that $h\in {\cal DH}$, a contradiction, since $A\neq {\bf R}$.

{\bf Case 2:}

Suppose that $f_{1}=g_{\alpha }\mbox{ with }\alpha <{\bf c}.$ Then the restriction of function $f_{2}(x)=h(x)-g_
{\alpha}(x)$ to Hamel basis $H$ is 
\begin{equation}
f_{2}|_{H}(x)=\left \{\begin{array}{ll}
                x_{\beta}-g_{\alpha}(t_{\beta ,i}),
                &\mbox{ if }x=t_{\beta ,i}\mbox{ for some }\beta <
             {\bf c}, i<\omega, \\
             z_{\beta}-g_{\alpha}(h_{\beta }),
             &\mbox{ if }x=h_{\beta} \not \in \{t_{\gamma ,
             i}|\gamma <{\bf c}, i<\omega \}.
             \end{array}
      \right.
\end{equation}

Since $f_{2}$ is linear and $y_{\alpha}$ does not belong to the ${\bf Q}$-vector space generated by numbers  $<x_{\beta}-g_{\alpha}(t_{\beta ,i}), z_{\beta}-
g_{\alpha}(h_{\beta})>_{\beta, i}$, it follows that $y_{\alpha}\not \in range(f_{2})$, therefore 
$f_{2}\not \in {\cal DH}$, contradicting our assumption about $h$. 

In conclusion, $h\in {\cal D}^{*}{\cal H}(A)\backslash ({\cal CH}+{\cal DH})$.
\end{proof} 

\begin{corollary} 
The inclusion ${\cal CH}+{\cal DH}\subset {\cal H}$ is strict. 
\label{foo2} 
\end{corollary} 

Putting all things together we obtain a different hierarchy for the 
additive case: 
\begin{corollary}
We have
\[
{\cal DH}=U\cdot {\cal DH}\subseteq QU\cdot {\cal DH}\subseteq {\cal CH}+{\cal DH}
\subset {\cal H}.
\]
\end{corollary}
\begin{proof}
We prove that the quasi-uniform limit of a sequence of additive Darboux functions is in ${\cal CH}+{\cal DH}$, then we apply Corollary~\ref{foo2}. 

Let $\{g_{n}\}_{n=1}^{\infty},\mbox{ } g_{n}\in {\cal DH}$, be a  sequence of functions
 quasi-uniformly convergent towards $g$. Clearly $g\in {\cal H}$.
 Let $\epsilon >0$ and consider the neighborhood $U$ of $0$
 and the index $n_{\epsilon}\geq 1$ s.t. $|g(x)-g_{n_{\epsilon}}(x)|
<\epsilon ,\forall x\in U$. Since $g-g_{n_{\epsilon}}$ is an additive function
bounded in a neighbourhood of 0, it is a continuous function, therefore 
$g=(g-g_{n_{\epsilon}})+g_{n_{\epsilon}}\in {\cal CH}+{\cal DH}$.
\end{proof}

\section{A personal recollection} 

If memory serves me well I first met professor Solomon Marcus in the spring of 1989. I was 19, in that year between high-school and college when one had to  complete nine-months military service, mandatory in communist Romania. The preceding year, still in high school, in a period of significant personal growth (I ranked first that year in the National Mathematical Olympiad in my age category), I became aware of the impending transition to a professional mathematician's career and started wondering what my (mathematical) future will be. 

I was, at that time, somewhat aware of the basics of college-level mathematical education. Through my father, a high school mathematics teacher, I had access to most of the standard college-level textbooks in use at the University of Bucharest. I (thought I) was reasonably well acquainted with first-year calculus and had been attempting to become familiar with the basics of complex analysis and its applications to analytic number theory. 

It wasn't meant for me to continue on that path. My scientific beginnings lie solidly  in Professor Marcus's range of scientific interests: While still in highschool I had read professor Marcus's books  \cite{marcus1967notiuni} and  \cite{marcus1975din}. I was struck by the (somewhat quaint nowadays) elegance and charm of the exotic properties in the theory of real functions, of the type investigated by Waclaw Sierpinski or Andrew Bruckner, in Romania by Simion Stoilow, Alexandru Froda and Professor Marcus himself. I became interested in generalizations of the intermediate value (a.k.a. Darboux) property, and sent professor Marcus a letter essentially containing what was to become later \cite{Ist1} one of my first scientific papers published (in 1991/92) outside Romania. 

More importantly, during the summer of 1989 I had made the (crucial, for my career) connection to theoretical computer science: by sheer hazard, in the summer of 1988 I had stumbled (in an used bookstore in my hometown, Gala\c{t}i) upon two research-level monographs on computability theory and formal languages written by two of Professor Marcus's disciples: Cris Calude \cite{calude1982complexitatea} and Gheorghe P\~{a}un \cite{paun1982gramatici}, and (thought I had) solved one open problem from each textbook. By the time I entered college I was the author of two scientific papers in theoretical computer science \cite{MR1087591,MR1100346} and the basic direction of my scientific path had pretty much been decided. 

When I entered the University of Bucharest in the fall of 1989 I was {\bf not} a  student of Professor Marcus: I simply was assigned to a different series which, on the other hand, had Cris Calude as professor. I continued to work in the research group of Marcus, Calude, P\~{a}un,  realizing soon enough that I was more interested in the research I was doing ``on the side" than in the (fairly different) courses I was taking in the Faculty of Mathematics. 

Communism fell, and with its fall came the possibility of unhindered publication outside Romania: Until 1994, the year I had left to pursue a Ph.D. in the United States I was an author of several papers \cite{determining-stationary,kn:CaludeIstrateZimand,istrate1994self,istrate1994some}. More importantly, with it came the possibility of working abroad. This had an unfortunate impact on the composition of the research group I belonged to: in 1992 Cris left for New Zealand. After getting a Ph.D. at the University of Bucharest (under the supervision of Cris) 
 Marius Zimand, one of the members of the Marcus group and coauthor on paper \cite{kn:CaludeIstrateZimand}, became in 1993 a Ph.D. student in Computer Science at the University of Rochester. I followed him in the fall of 1994, among the first in what later became a mass exodus of Romanian scientists. 

As I was completing my undergraduate studies at the University of Bucharest I wrote a  thesis \cite{Ist3} on the theory of real functions under the guidance of Professor Marcus. Its purpose was to create a survey (based on the literature available in  pre-internet 1994 Romania) of the theory of generalized versions of the Darboux property. More importantly, it contained some original results, some of them (eventually) published in Real Analysis Exchange \cite{Ist1,Ist2}. A few of these results had, however, remained confined to my thesis.  

More than 15 years after my college graduation (as Professor Marcus visited Timi\c{s}oara as an invited speaker of SYNASC 2010) I was shocked to find out that he still remembered that some of the original results in my thesis had not been published. He suggested that I should revisit them, and eventually submit them for publication. 

The purpose of this paper is to graciously answer this request. It is fitting to dedicate them to a man that has influenced my career in so many profound ways. Professor Marcus has undoubtedly had many students and disciples. While my American experience led my research career on a set of paths very different from those of my scientific beginnings (a recent sample is \cite{kneser}), I would like to point out that I am probably one of the few students of Professor Marcus to have worked in several of his major research directions: real analysis \cite{Ist1,Ist2}, formal languages \cite{MR1087591, istrate-etol}, combinatorics on words \cite{istrate1994self, istrate1994some}, recursive function theory \cite{determining-stationary,kn:CaludeIstrateZimand}. This paper adds to this heritage. 

I was extremely privileged to have been part of the circle of people influenced by Professor Marcus during these more than 25 years. 
To him, all my gratitude and my best wishes. 

\section*{Conclusions and Acknowledgment} 

Some of the concepts in this paper would deserve, we believe, further investigations, e.g. the (two versions of) {\bf Q}-differentiability. For instance, a natural question is whether the {\bf Q}-derivatives of a {\bf Q}-differentiable function satisfy some weak version of Darboux property. Another question raised by Theorem~\ref{diff} is whether one could somehow represent all functions that are differentiable with respect to some fixed additive function $H$ in terms of $H$ and ordinary differentiable functions only. 

As for the second part of the paper, the original statement of Theorem~\ref{two} in our thesis \cite{Ist3} did not contain the condition that the {\bf Q}-vector space $A$ has cardinality ${\bf c}$. We believe that this condition is not needed, but couldn't see how to remove it from the hypothesis of Theorem~\ref{two}. We leave this issue as an open problem. 

Writing this paper has been supported in part by CNCS IDEI Grant PN-II-ID-PCE-2011-3-0981 ``Structure and computational difficulty in combinatorial optimization: an interdisciplinary approach".



 { \footnotesize  
\medskip
\medskip
 \vspace*{1mm} 
 
\noindent {\it Gabriel Istrate}\\  
West University of Timi\c{s}oara and the e-Austria Research Institute\\
Bd. V. P\^{a}rvan 4, cam 047,\\
Timi\c{s}oara, RO-300223, Romania. \\
E-mail: {\tt gabrielistrate@acm.org}\\ \\  
 }

 \end{document}